\documentclass{article}
\usepackage{graphicx}

\usepackage[fleqn]{amsmath}
\usepackage{amsthm}
\usepackage{amsmath}
\usepackage{amssymb}
\usepackage{amsbsy}
\usepackage{amsfonts}
\usepackage{mathrsfs}
\usepackage[all]{xy}
\usepackage{amstext}
\usepackage{amscd}
\usepackage[dvips]{epsfig}
\usepackage{psfrag}
\usepackage{enumerate}
\usepackage{flafter}
\allowdisplaybreaks

\theoremstyle{plain}
\newtheorem{thm}{Theorem}[section]
\newtheorem{prop}[thm]{Proposition}
\newtheorem{cor}[thm]{Corollary}
\newtheorem{lem}[thm]{Lemma}
\theoremstyle{definition}
\newtheorem{exa}[thm]{Example}
\newtheorem{conj}[thm]{Conjecture}
\newtheorem{rem}[thm]{Remark}
\newtheorem{defn}[thm]{Definition}

\def\Ker{\mathop{\mathrm{Ker}}\nolimits}

\newcommand{\lra}{\longrightarrow}
\newcommand{\ra}{\rightarrow}
\newcommand{\Q}{{\Bbb Q}}

\newcommand{\Z}{{\Bbb Z}}
\newcommand{\N}{{\Bbb N}}

\newcommand{\C}{\mathbb{C}}

\newcommand{\pc}[2]{\mbox{$\begin{array}{c}
\includegraphics[scale=#2]{#1.eps}
\end{array}$}}

\begin{document}
\large
\begin{center}
{\bf\Large Meta-nilpotent quotients of mapping-torus groups}\end{center} \begin{center}{\bf\Large and two topological invariants of quadratic forms}
\end{center}
\vskip 1.5pc
\begin{center}{Takefumi Nosaka\footnote{
E-mail address: {\tt nosaka@math.titech.ac.jp}
}}\end{center}
\vskip 1pc
\begin{abstract}\baselineskip=12pt \noindent
We determine the center of a meta-nilpotent quotient of a mapping-torus group. As a corollary, we introduce two invariants, which are quadratic forms, of knots and of mapping classes.
\end{abstract}

\begin{center}
\normalsize
\baselineskip=17pt
{\bf Keywords} \\
\ \ \ Meta-nilpotent groups, knots, Alexander polynomial, Johnson homomorphism \ \ \
\end{center}
\large
\baselineskip=16pt

\section{Introduction}
For a free group $F$, let $F_1$ be $F$, and let $F_k$ be the commutator subgroup $[F,F_{k-1}]$ inductively. We suppose an automorphism $\tau : F \ra F$, which admits the semi-direct product $ F \rtimes \Z$. By meta-nilpotent quotient, we mean the quotient $ F/F_k \rtimes \Z$. This paper mainly discusses the case $k=3$. 

In some cases, the group $ F \rtimes \Z$ is topologically constructed as follows.
Let $\mathcal{M}_{g,r}$ be the mapping class group of the oriented compact surface $\Sigma_{g,r }$ with $r$-boundaries, which fix the boundary pointwise. 
Given a mapping class $f \in \mathcal{M}_{g,r} $, {\it the mapping-torus} $T_f$ is the quotient space of $\Sigma_{g,r} \times [0,1 ] $ subject to $ (x,0) \sim (f(x),1)$ for any $x \in \Sigma_{g,r}$; then, the fundamental group $ \pi_1( T_f)$ is the semi-direct product $ \pi_1( \Sigma_{g,r}) \rtimes_f \Z $. 
Since the Dehn-Nielsen theorem (see, e.g., \cite{FM}) claims a natural injection $ \mathcal{DH}:   \mathcal{M}_{g,1} \hookrightarrow \mathrm{Aut}(\pi_1(\Sigma_{g,1} ))$, 
the study of the difference between the image $\mathrm{Im} \mathcal{DH} $ and $ \mathrm{Aut}(\pi_1(\Sigma_{g,1} ))  $ is interesting and has been investigated in some ways (e.g., in terms of $ F \rtimes \Z$); see, e.g., \cite{FM,Joh,Ki,Sato1} and references therein. 
.


In this paper, we give an observation of the difference in terms of centers. 
However, some groups appearing in the study of $\mathcal{M}_{g,r}$ have trivial centers in many cases. For example, $ \mathcal{M}_{g,0} $ has no center if $g \geq 3$ (see \cite[\S 3.4]{FM}), and $ \pi_1( \Sigma_{g,r}) \rtimes \Z $ and its metabelianization are mostly centerless. 
It is often that, if a mapping class group $ \mathcal{M}_{g,r} $ with $r>1$ has non-trivial center, 
the central elements can be described in an easy way. 
For these reasons, few invariants of mapping classes and mapping-tori are studied so far by using centers.

In contrast, this paper focuses on the meta-nilpotent quotient $F/F_3 \rtimes \Z $ and determines the center (Theorem \ref{gggg2}). Moreover, if $F \rtimes \Z $ is a mapping-torus group $ \pi_1( \Sigma_{g,1}) \rtimes_f \Z $, we show (Corollary \ref{gggg23}) that the center of $F/F_3 \rtimes \Z $ always contains $\Z^g$; further, we explicitly express the central elements of the summand $\Z^g$ in some cases (Proposition \ref{gggg4}). Here, the point is to express explicitly the action of $ \Z$ on $F/F_3$ (Proposition \ref{gggg3}).

In application, we construct two invariants which are quadratic forms and are valued in the centers of the groups $F/F_3 \rtimes \Z $. First, we define a quadratic form from a mapping class $[f] \in \mathcal{M}_{g,1}$; see Definition \ref{b3440}. Roughly speaking, this quadratic form is defined by a correspondence from the homology $ H_1(F/F_3 \rtimes \Z ;\Z[t^{\pm 1}] )$ to a boundary element of $ \pi_1( \Sigma_{g,1}) \rtimes_f \Z $. On the other hand, as an analogy, we give a knot invariant of a quadratic form over $\Q$ in terms of evaluating longitudes from $F/F_3 \rtimes \Z $ (Definition \ref{bbds3}). We also contrast our invariant with the Blanchfield pairing \cite{Bla}, which is a classical invariant of hermite bilinear form (see Proposition \ref{lem0008}).

This paper is organized as follows. Section \ref{S54} states the theorems, and Sections 3 and 4 describe the two invariants of quadratic forms. Sections \ref{s23} and \ref{k2k2} give the proofs of the statements in Sections 3--4.

\section{The main theorems}
The purpose of this section is to state Theorems \ref{gggg2} and \ref{gggg4}.
\label{S54}
\subsection{Center of the meta-nilpotent quotient of the free group}
Consider the situation that the free group, $F$, of rank $m$ is acted on by $\Z= \{ \tau^{\pm n } \}_{n\in \Z} $. Then, we have the semi-direct product $ F/F_k \rtimes \Z $ for any $k$. In this paper, we will focus on the case $k=3$ and analyze the center of $ F/F_3 \rtimes \Z $ (see Theorem \ref{gggg2}).

In order to state Theorem \ref{gggg2}, we regard the quotient action $ \tau : F/F_2 \ra F/F_2 $ as an $(m\times m)$ matrix over $\Z $, since $F/F_2 \cong \Z^m $.
The characteristic polynomial of $\tau$ is called {\it the Alexander polynomial (of $\tau$)} and is denoted by $\Delta_{\tau}(t ) \in \Q[t^{\pm 1}]$.
Fixing an inclusion $ \Q \hookrightarrow \mathbb{C}$, let $ \alpha_1, \dots, \alpha_k \in\mathbb{C}$ be the distinct roots of $\Delta_{\tau}(t)$.
We further consider the complexification $F/F_2 \otimes \C $ and regard it as a $\C[ t^{\pm 1} ]$-module; elementary divisor theory gives a $\C[ t^{\pm 1} ]$-isomorphism
\begin{equation}\label{a5} F/F_2 \otimes \C \cong \bigoplus_{i=1}^k \Bigl( \frac{\C[t^{\pm 1} ]}{(t- \alpha_i )^{n_{1}^{(i)} } } \oplus \frac{\C[t^{\pm 1} ]}{(t- \alpha_i )^{n_{2}^{(i)} } } \oplus \cdots \oplus \frac{\C[t^{\pm 1} ]}{(t- \alpha_i )^{n_{\ell_i}^{(i)} } } \Bigr) \end{equation}
for some natural numbers $\ell_1,\dots, \ell_k $ and $ n_{1}^{(i)},\dots ,n_{\ell_i }^{(i)} $ with
$ 0< n_{1}^{(i)} \leq n_{2}^{(i)} \leq \cdots \leq n_{\ell_i }^{(i)} $.

We determine the center of $ F/F_3 \rtimes \Z$ as follows:
\begin{thm}\label{gggg2}
Let $F$ be the free group of rank $m < \infty $. Suppose $\Delta_{\tau}(\pm 1) \neq 0.$
Then, the center of $F/F_3 \rtimes \Z $ is a free $\Z$-module, and the rank is equal to the following sum:
\begin{equation}\label{a77} \sum_{ (i,j) \in \{ (i,j ) \in \N^2 | \ \alpha_i \alpha_j =1 , \ i< j \} } \sum_{ u=1}^{\ell_i }\sum_{ v=1}^{\ell_j } \mathrm{min} ( n_u^{ (i)} , n_v^{ (j)}).
\end{equation}
\end{thm}
The proof appears in \S \ref{s23}. Here, as examples, we give corollaries in special situations, which will be used in Corollary \ref{gggg45533} and \S\S \ref{dsa}--\ref{dsa2}. 
\begin{cor}\label{gggg23}
With the notation in Theorem \ref{gggg2}, we now assume $\ell_1 = \cdots = \ell_k=1$ and
$n_1^{(1)} = \cdots = n_1^{(k)}=1$.
Then, the rank of the center of $F/F_3 \rtimes \Z $ is equal to the cardinality $\#
\{ (i,j ) \in \N^2 | \ \alpha_i \alpha_j =1 \} .$
\end{cor}
\begin{cor}\label{gggg23}
With the notation in Theorem \ref{gggg2}, 
assume the existence of $g \in \mathbb{N}$ with $k=2g$  and the reciprocity, i.e.,
$ \alpha_i= \alpha_{2g -i}^{-1}$ and $\ell_{i }=\ell_{2g-i} $ and $ n_u^{ (i)}= n_u^{(2g-i) }$ for all $i $ and $u \in \mathbb{N}.$
Then, the rank of the center of $F/F_3 \rtimes \Z $ is larger than or equal to $m/2$.

Furthermore, if every $\ell_i$ is equal to 1,
then the center of $F/F_3 \rtimes \Z $ is isomorphic to $\Z^{m/2}=\Z^g. $
\end{cor}
\begin{proof}
By reciprocity, the sum \eqref{a77} is smaller than the sum $ \sum_{ i=1} ^{g} \ell_i$ which is equal to $ m/2 .$ Regarding the latter part,
the sum \eqref{a77} is $ \sum_{ i=1} ^{g} n_1^{(i)}$, which is equal to $ m /2$ by \eqref{a5}.
\end{proof}

\subsection{Central elements in rationalization of the $\Z$-action.}
\label{s22}
Now let us study a rationalization of the group $ F/F_3 \rtimes \Z $ and explicitly describe some central elements of $ F/F_3 \rtimes \Z $ (Theorem \ref{gggg4}).

First, we describe the group structure of $F/F_3$. Since $F/F_2 \cong \Z^{m} $ and $F_2/F_3 \cong \Z^{m(m-1)/2 }$, we choose a basis $\{ e_{i}\}_{i=1}^{m} $ of $F /F_2$ and a basis $\{ e_{i j }\}_{ 1 \leq i <j \leq m} $ of $F_2 /F_3$. As an analogy to the exterior product, we define a bilinear map 
$$ \sqcap: F/F_2 \times F/F_2 \ra F_2/F_3$$
 by setting 
$$e_{i} \sqcap e_j =e_{ij} \ \ \ \mathrm{and} \ \ \ e_{j} \sqcap e_i = e_i \sqcap e_i = 0 \ \ \mathrm{for }  i< j.$$
Accordingly, we have a group operation on $F/F_2 \times F_2 /F_3$ given by
\begin{equation}\label{hhh} (a , \alpha ) \cdot (b , \beta ) :=(a +b, \ a \sqcap b + \alpha + \beta), \ \ \
\end{equation}
for $a,b \in F/F_2, \ \alpha ,\beta \in F_2/F_3$. It can be seen that this group is isomorphic to $F/F_3$.

Next, let us examine the rationalization. Consider the $\Q$-extension of $\sqcap$, i.e., $ \sqcap: F/F_2 \otimes \Q \times F/F_2 \otimes \Q \ra F_2/F_3 \otimes \Q $. Then, the operation \eqref{hhh} makes $ F/F_2 \otimes \Q \times F_2/F_3 \otimes \Q $ into a group. We denote the group by $ F/F_3 \otimes \Q$, and we later address the group structure of $ F/F_3 \otimes \Q$ in details (see Proposition \ref{gggg3}).

We will analyze the center of $(F/F_3 \otimes \Q)\rtimes \Z $ in Theorem \ref{gggg4}.
For this,
the following easy lemma implies that
we may focus on the elements in $Z(e ,1) \cap ( F_2 /F_3 \otimes \Q \times \{ 0\})$.
\begin{lem}\label{lem24}
Let $\Z$ act on a group $G$ with unit $e$, and consider the semi-direct product $G \rtimes \Z$. Let $Z(G) \subset G$ be the center of $G$ and $Z(e,1)$ be the centralizer subgroup of $(e,1) \in G \rtimes \Z$. Then, the center of $G \rtimes \Z $ is equal to the intersection $Z(e,1) \cap (Z(G) \times \{ 0\})$.
\end{lem}
\begin{proof}
Any central element $C \in G \rtimes \Z$ commutes with $ (e,1) $ and $ (x,0)$ for any $x \in Z(G)$. Therefore, $C \in Z(e,1) \cap ( Z(G) \times \{ 0\})$. Conversely, by the definition of $G \rtimes \Z$, any $C \in Z(e,1) \cap (Z(G) \times \{ 0\})$ is obviously central.
\end{proof}

Now, we will give an expression of central elements, when $m=2g$ for some $g \in \mathbb{N}$. Here, we further assume that
$k=1$, i.e., $ F/F_2 \otimes \Q \cong \Q[t ]/{f_1}(t)$, and 
$  f_1(t^{-1})= t^{-\mathrm{deg}f_1} f_1(t) $. Then, if we expand $f_1$ as $ \sum_{i=0}^{2g} a_it^{i}$ with $a_0=a_{2g}=1$, we have $ a_j=a_{2g- j}$ for any $0< j\leq g$. For any $i<j \leq 2g $ and $\ell \leq g$, let us introduce $d_{i,j }^{(\ell)}\in \Z$ as follows.
Let $ d_{1 ,j }^{(\ell)}\in \Z $ be the Kronecker delta $ \delta_{j-1,\ell}$ for $ j \leq g $, and $ d_{1, j }^{(\ell)}$ be $ \delta_{2g-j+1 ,\ell}$ for $ j > g $. By double induction on $i$ and $j $, we define $d_{i,j }^{(\ell)}\in \Z$ by using the formulas
\begin{equation}\label{ooo}d_{i,j}^{(\ell)} := \begin{cases}
d_{i-1,j-1 }^{(\ell)}-a_{j-1} \delta_{i-1,\ell}+ a_{j-2} \delta_{j-1,\ell}, & \mathrm{if} \ \ i \leq g, j\leq g, \\
d_{i-1,j-1 }^{(\ell)}-a_{j-1} \delta_{i-1,2g-\ell}+ a_{j-2} \delta_{j-1,\ell} ,& \mathrm{if} \ \ i > g, j\leq g ,\\
d_{i-1,j-1 }^{(\ell)}-a_{j-1} \delta_{i-1,\ell}+ a_{j-2} \delta_{j-1,2g-\ell} ,& \mathrm{if} \ \ i \leq g, j>g, \\
d_{i-1,j-1 }^{(\ell)}-a_{j-1} \delta_{i-1,2g-\ell}+ a_{j-2} \delta_{j-1,2g-\ell}, & \mathrm{if} \ \ i > g, j> g. \\
\end{cases} \end{equation}
\begin{thm}\label{gggg4}
Let $m=2g$ for some $g \in \mathbb{N}$. Suppose all $\ell_i=1$ in \eqref{a5} and the reciprocity $ f_1(t^{-1})= t^{-\mathrm{deg}f_1} f_1(t)$ (possibly $f_1(\pm 1) =0$) as above. Then, the element $C_\ell$ defined to be $ ( \sum_{i<j} d_{i,j}^{(\ell)} e_{i,j},0) \in( F_2/F_3 \otimes \Q) \rtimes \Z $ satisfies $\tau(C_\ell)=C_\ell$. In particular, $C_\ell$ lies in the center of $ ( F_2/F_3 \otimes \Q) \rtimes \Z$.

Moreover, the center is a $\Q$-vector space of dimension $g$, and spanned by the elements $C_1, \dots, C_g$. 
\end{thm}
The proof will appear in \S \ref{s23}. 

\begin{exa}
If $g=1$, then $ C_1 := ( e_{12},0)$.

Next, if $g=2$, the central elements $C_1$ and $C_2$ are given by
\[ C_1 = (e_{12} +e_{14}+(1-a_2)e_{23}+ e_{34},0), \ \ \ \ \ \ \ C_2 = ( e_{13} +a_1e_{23}+ e_{24},0). \]
If $g=3$, the central elements $C_1,C_2$, and $C_3$ are given by
\[C_1=( e_{12}+e_{16} +(1-a_2)e_{23} -a_3 e_{24}-a_4 e_{25} +(1-a_2) e_{34}-a_3 e_{35}+(1-a_4)e_{45}+e_{56},0) ,\]
\[C_2=( e_{13}+ e_{15}+a_1 e_{23} +e_{24} +a_1 e_{25} +e_{26}+(a_1 -a_3) e_{34}+e_{35}+a_1 e_{45}+e_{46},0), \]
\[C_3= (e_{14}+ a_1 e_{24} +e_{25} +a_2 e_{34} +a_1 e_{35}+ e_{36},0).\]
\end{exa}

Theorem \ref{gggg4} also integrally holds in some cases. Precisely, we have
\begin{cor}\label{gggg45533}
Regarding $F/F_2$ as a $\Z[t^{\pm 1 } ]$-module, we assume that $F/F_2 $ is isomorphic to $ \Z[t^{\pm 1}] / (1 +a_1 t + \cdots +a_{2g-1}t^{2g-1}+t^{2g} )$ with $a_{2g-j}=a_j$, as in the reciprocity in Theorem \ref{gggg4}. Then, the center of $F/F_3 \rtimes \Z$ is spanned by the central elements $C_1,\dots, C_{ g}$ in Theorem \ref{gggg4}.
\end{cor}
\begin{proof}
We can verify that $C_\ell$ lies in $F_2/F_3$ by definition, and satisfies $ \tau (C_\ell) = C_{\ell}$ by Theorem \ref{gggg4}. Hence, $C_{\ell}$ lies in the center of $F/F_3 \rtimes \Z$. By Corollary \ref{gggg23}, the rank of the center is $g$.
Since the $e_{1,j}$-th component of $C_{\ell}$ is $d_{1,j}=\delta_{j-1,\ell}$ for $j \geq g$, the elements $ C_1, \dots, C_g$ are linearly independent. Thus, the center is spanned by $C_\ell$'s.
\end{proof}

\section{Knot invariants of quadratic forms}
\label{dsa}
As an application of \S \ref{S54}, we will introduce a knot invariant of a quadratic form (The proof of every proposition in this section will be shown in \S \ref{k2k2}). Let $K \subset S^3$ be a tame knot. We suppose basic knowledge about knot theory as seen in \cite{CF,CL,Hil}. We denote the fundamental group $\pi_1(S^3 \setminus K) $ by $ \pi_K$ hereafter. Furthermore, we fix a meridian $ \mathfrak{m} \in \pi_K$.

Here, let us briefly review the (rational) Alexander modules and polynomials of knots. Since $H_1( S^3 \setminus K) \cong \Z $, we have the universal abelian covering $\widetilde{E}_K \ra S^3 \setminus K $. This fundamental group $ \pi_1( \widetilde{E}_K)$ is the commutator subgroup $[\pi_K ,\pi_K]$, which is acted on by the covering transformation $t$. The rational homology $H_1( [\pi_K ,\pi_K ] ;\Q ) $ as a $\Q[t^{\pm 1}]$-module is called {\it the (rational) Alexander module} of $K$. We then have $\Q[t^{\pm 1}]$-module isomorphisms,
\begin{equation}\label{pp87}H_1( \widetilde{E}_K ;\Q ) \cong H_1( [\pi_K ,\pi_K ] ;\Q ) \cong \Q[t ]/{f_1}(t) \oplus \cdots \oplus \Q[t ]/{f_n}(t) , \end{equation}
for some non-zero polynomials ${f_1}(t), \dots, {f_n}(t) $ such that $f_{i+1}(t)$ is divisible by $ f_i(t)$. It is well known as the reciprocity (see, e.g., \cite[Section IX]{CF}) that 
$$f_i(t^{-1})= t^{-{\rm deg} f_i } f_i (t) \ \ \ \mathrm{and} \ \ \ \ f_i(\pm 1) \neq 0. $$
The product ${f_1}(t) \cdots {f_n}(t) $ is equal to the usual Alexander polynomial $\Delta_K $ of the knot $K$.

Next, we will consider the set \eqref{bbb5} below. Let $F$ be the free group of rank $ \mathrm{deg}(f_n(t))$, which is even by reciprocity. Then, we get uniquely the isomorphism $t : F/F_2 \otimes \Q \ra F/F_2 \otimes \Q$ which admits the $\Q[t^{\pm 1}]$-isomorphism $F/F_2 \otimes \Q \cong\Q[t ]/{f_n}(t) . $ In addition, since $ \Z=\langle \mathfrak{m}\rangle $ has a conjugacy action on $\pi_K$, let us consider the following set consisting of group homomorphisms:
\begin{equation}\label{bbb5} \mathcal{H}_K := \{ \ \textrm{homomorphism } f : \pi_1 (S^3 \setminus K) \ra (F/F_2 \otimes \Q) \rtimes \Z \mathrm{ \ s.t. \ } f( \mathfrak{m})=(0,1) \ \}.
\end{equation}
This set is a $\Q[t^{\pm 1}]$-module and is identified with the set of $\Z$-equivariant homomorphisms from $[\pi_K, \pi_K] $ to $F/F_2 \otimes \Q $. Hence, by \eqref{pp87}, the $\Q[t^{\pm 1}]$-module $ \mathcal{H}_K $ is isomorphic to the $\Q[t^{\pm 1}]$-module $H_1( [\pi_K ,\pi_K ] ;\Q ) $; in particular, the isomorphism class of $\mathcal{H}_K $ is independent of the choice of $\mathfrak{m}.$

Now we are in a position to state Proposition \ref{lem6}. Let $F/F_k \otimes \Q$ be the rationalization of the nilpotent group $F/F_k$, i.e., a Malcev completion of $F/F_k$, and suppose an automorphism $\tau : F/F_k \otimes \Q \ra F/F_k \otimes \Q$, as a lift of $t$. Then, we have the semi-direct product $ (F/F_k \otimes \Q) \rtimes \Z $. We will discuss lifts of $ \mathcal{H}_K$ to $ (F/F_k \otimes \Q) \rtimes \Z $:

\begin{prop}\label{lem6}
Let $k \in \N $. Every homomorphism $ f : \pi_1(S^3 \setminus K) \ra (F/F_2 \otimes \Q) \rtimes \Z$ in $\mathcal{H}_K$ uniquely admits a lift $ \widetilde{f} : \pi_1(S^3 \setminus K) \ra (F/F_k \otimes \Q) \rtimes \Z $ of $f$ such that $ \widetilde{f} ( \mathfrak{m})=(0,1)$.
\end{prop}
Next, we will focus on the case $k=3$ and discuss the preferred longitude evaluated by such a lift:

\begin{prop}\label{lem7}
Let $\mathfrak{l} \in \pi_K $ be the preferred longitude.

(I) Given a lift $\widetilde{f} : \pi_1 (S^3 \setminus K )\ra (F/F_3 \otimes \Q) \rtimes \Z $, the target $\widetilde{f} (\mathfrak{l}) $ is contained in the center of $ ( F/F_3 \otimes \Q) \rtimes \Z $.

(II) This $\widetilde{f} (\mathfrak{l} ) $ does not depend on the choice of the lift $ \tau : F/F_3 \otimes \Q \ra F/F_3 \otimes \Q $ of $t$.
\end{prop}
To summarize, we can introduce a new knot invariant:
\begin{defn}\label{bbds3}
We define a knot invariant, $ \mathcal{Q}_K$, to be the map from $ \mathcal{H}_K \cong H_1(\widetilde{E}_K ;\Q)$ to the center $Z (F/F_3 \otimes \Q \rtimes \Z) $ which sends $f$ to $\widetilde{f} (\mathfrak{l} ) $.
\end{defn}
If $n=1$ in \eqref{pp87}, the computation of this invariant is not so hard; in fact, since we later describe concretely the group structure of $(F/F_3 \otimes \Q)\rtimes \Z$ (see Propositions \ref{gggg3}), it is not so hard to find $\widetilde{f} $ and $ \widetilde{f} (\mathfrak{l} ) $ by using a presentation of $\mathfrak{l} $; see Examples \ref{ex5} and \ref{ex6}.

Furthermore, we later show that this invariant is of a quadratic form and isometric:
\begin{prop}\label{lem8}
This map, $ \mathcal{Q}_K $, is a quadratic form on $\mathcal{H}_K $. Moreover, $\mathcal{Q}_K $ is isometric in the sense of $\mathcal{Q}_K (\tau_*(f))= \mathcal{Q}_K (f)$ for any $f \in \mathcal{H}_K $.
\end{prop}
Such a quadratic form is completely classified in \cite{Mil6}. Thus, after the computation of $\mathcal{Q}_K $, we can obtain quantitative information from the invariant $\mathcal{Q}_K $.

\subsection{Some examples of the invariant $\mathcal{Q}_K $ }
\label{jiji}
We will explain a procedure of computing $\mathcal{Q}_K $. In many cases, given a diagram of a knot $K$ with the Alexander polynomial $\Delta_K$, it is not hard to compute $\mathcal{Q}_K $, as shown in what follows.
\begin{exa}\label{exa8}
Let us compute the invariant $\mathcal{Q}_K $ of the trefoil knot. Notice that the Wirtinger presentation gives a presentation of the fundamental group:
$$ \pi_K \cong \langle\ a,b ,c\ | \ c = b^{-1}ab, \ a = c^{-1}bc, \ b = a^{-1}ca \ \rangle \cong \langle\ a,b \ | \ aba=bab \ \rangle. $$
As is well-known, the Alexander polynomial is $\Delta_K= t^2 -t +1$ and $H_1(\widetilde{E}_K;\Z) \cong \Z[t] / ( t^2 -t +1)$. Choose an automorphism $\tau: F/F_3 \otimes \Q \ra F/F_3\otimes \Q $ such that $F/F_2 \otimes \Q \cong \Q[t]/(t^2-t+1)$.

Then, for $x,y,x',y'\in \Z$, the correspondence,
$$ a \longmapsto (0 ,1) , \ \ \ b \longmapsto (x +t y,1) , \ \ \ c \longmapsto (x' + ty',1) , $$
gives rise to a homomorphism $f: \pi_K \ra F/F_2\rtimes \Z$, if and only if $x' + y' t= x+y -t x \in F/F_2 .$ Moreover, if so, we can verify that the correspondence,
$$ a \longmapsto ( 0,0 ,1) \in F_2/F_3 \times F/F_2\times \Z , \ \ \ b \longmapsto (\beta, x +t y,1) \in F_2/F_3 \times F/F_2\times \Z, $$
yields a homomorphism $f: \pi_K \ra F/F_3\rtimes \Z$ if and only if $\beta = ( x -x^2 -y +2x y +y^2 )/2 + \alpha. $ Here, $\alpha$ is arbitrary. Then, the Wirtinger presentation gives the preferred longitude $\mathfrak{l}$ as $ a^{-1}bc^{-1}a b^{-1 }c$. Thus, it is easy to compute $\widetilde{f}(\mathfrak{l}) $ as
$$ \widetilde{f}(\mathfrak{l})=\widetilde{f}(a^{-1}bc^{-1}a b^{-1 }c)= ( x ^2 - xy +y ^2 , 0,0) \in F_2/F_3 \times F/F_2\times \Z . $$
In summary, the map $ \mathcal{Q}_K : \mathcal{H}_K \ra Z (F/F_3 \otimes \Q \rtimes \Z)\cong \Q $ is equal to the map $(x +y t ) \mapsto x^2 -xy + y^2$.
\end{exa}
\begin{exa}\label{ex5}
In a similar way, we can compute $\mathcal{Q_K}$ when $\mathcal{H}_K$ is isomorphic to $\Q[t]/(1+at+t^2)$ for some $a \in \Q$. For example, when $K$ is one of the knots $4_1$, $5_2$ or $6_1$, the resulting computations are described as:
\[ \mathcal{Q}_{4_1} : \Q[t]/(t^2-3t+1) \lra \Q ; \ \ \ (x +t y)\longmapsto x^2 -3xy + y^2. \]
\[ \mathcal{Q}_{5_2} : \Q[t]/(2t^2-3t+2) \lra \Q ; \ \ \ (x +t y)\longmapsto 2x^2 -3xy + 2y^2. \]
\[ \mathcal{Q}_{6_1} : \Q[t]/(2 t^2-5t+2) \lra \Q ; \ \ \ (x +t y)\longmapsto 2x^2 -5xy + 2y^2. \]
\end{exa}
\begin{exa}\label{ex6}
However, the computation becomes little complicated even if $\mathrm{ deg}\Delta _K(t) = 4$, in which case we should use a computer program. Here, notice from Corollary \ref{gggg23} that the center of $(F/F_3 \otimes \Q) \rtimes \Z$ is isomorphic to $\Q^2$. As examples of $\mathrm{ deg}\Delta _K(t)= 4$, if $K$ is one of $5_1$, $ 6_2$ or $6_3$, we give the resulting computations of $ \mathcal{Q}_K$ as follows.
\[ \mathcal{Q}_{5_1} : \Q[t]/(t^4-t^3 +t^2-t+1) \lra \Q^2 ;\]
\[ (x +t y +t^2 z + t^3 w)\longmapsto (x^2 +xy+y^2 +wz+yz+z^2+ w^2 , -wx +wy +xy +wz+xz+yz). \]
\[ \mathcal{Q}_{6_2} : \Q[t]/(t^4-3t^3 +3t^2-3t+1) \lra \Q^2 ; \]
\small
\[ (x +t y +t^2 z + t^3 w)\longmapsto (w^2 +2wx+ x^2 +xy+y^2 +wz+yz+z^2, -2w^2 +wx-2x^2 +3wy -xy -2y^2 -wz+3xz -yz - 2 z^2).\]
\large
\[ \mathcal{Q}_{6_3} : \Q[t]/(t^4-3t^3 +5t^2-3t+1) \lra \Q^2 ; \]
\[ (x +t y +t^2 z + t^3 w)\longmapsto (-w^2+6wx-x^2+2wy-xy-y^2-wz+2xz-yz-z^2 ,\]
\[ \ \ \ \ \ \ \ \ \ \ \ \ \ \ \ \ \ \ \ \ \ \ \ \ \ \ \ \ \ \ \ \ \ \ \ \ 2w^2-9wx+2x^2-wy+3xy+2y^2+3wz-xz+3yz+2z^2). \]
\large
\end{exa}
In our experience, we conjecture a condition of non-degeneracy:
\begin{conj}
Let $C_1 ,\dots, C_g$ be the central elements in Theorem \ref{gggg4}. Then, for any $i \leq g$, the composite map $ \mathcal{H}_K \stackrel{\mathcal{Q}_{K} }{\lra} Z (F/F_3 \otimes \Q \rtimes \Z) \stackrel{\rm proj}{\lra}\Q \langle C_i \rangle $ would be non-degenerate. That is, every (real) eigenvalue of $ \mathcal{Q}_{K}$ is not zero.
\end{conj}

For comparison, let us examine the Blanchfield pairing \cite{Bla}, which is a hermitian bilinear form on $H_1( \widetilde{E}_K ;\Z)$. This pairing has been studied in a number of ways, and it has had some topological applications; see \cite{Hil} and references therein. If two knots have the same Blanchfield pairing up to equivalence, they are called $S$-{\it equivalent}. We will show the difference between the pairing and our quadratic form $\mathcal{Q}_K$. 
\begin{prop}\label{lem0008}
There are two $S$-equivalent knots such that the $\mathcal{Q}_K$ are different.
\end{prop}
\begin{proof}
Consider the Pretzel knots $P(3,3,-3)$ and $P(9,3,-3)$ depicted in Figure \ref{fig.color2}. According to \cite[Lemma 5.6]{BFKP}, they are $S$-equivalent. With the help of a computer, we can compute the quadratic forms as
\[ \mathcal{Q}_{ P(3,-3,3)}; \Q[t]/(2t^2 -5t+2) \lra\Q; \ \ (x+ty)\longmapsto 12 x^2 + 30xy+12y^2 , \]
\[ \mathcal{Q}_{ P(3,-3,9)}; \Q[t]/(2t^2 -5t+2) \lra\Q; \ \ (x+ty)\longmapsto 6x^2 + 15 x y+6y^2 .\]
Therefore, $\mathcal{Q}_{ P(3,-3,3)}=2 \mathcal{Q}_{ P(3,-3,9)} $. However, 2 admits no square-root in the field $\Q[t]/(2t^2 -5t+2) $. Hence, the two forms are not equivalent.
\end{proof}

\begin{figure}[htpb]
\begin{center}
\begin{picture}(100,50)
\put(-262,13){\pc{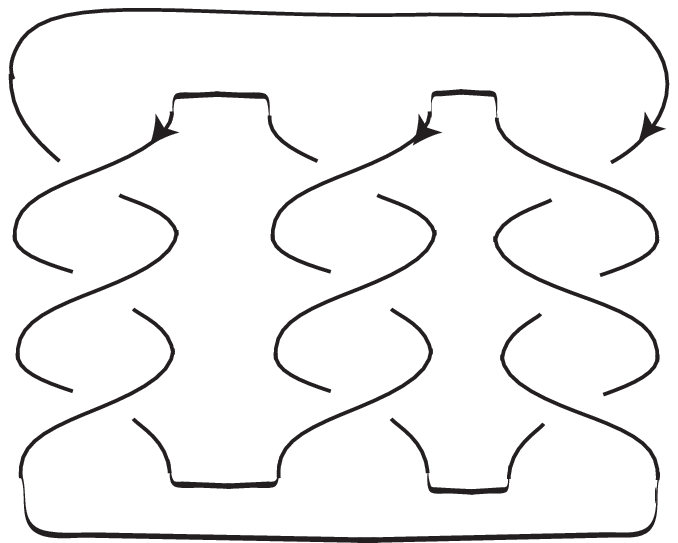}{0.4530174}}
\put(-112,13){\pc{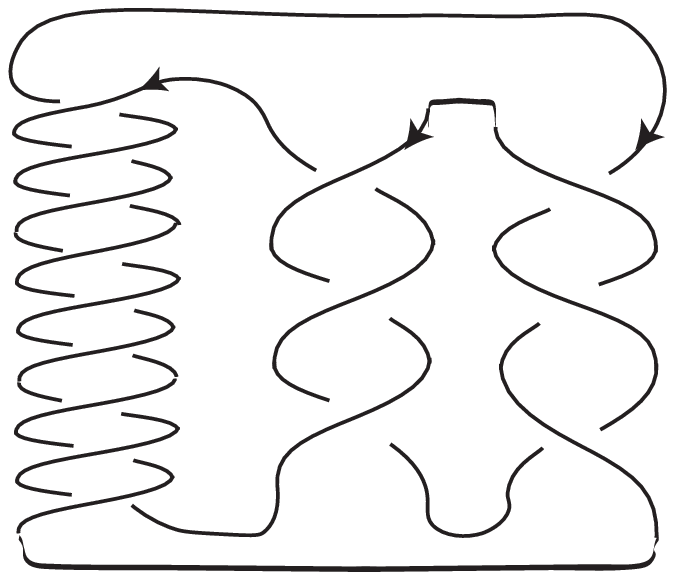}{0.4530174}}

\end{picture}
\end{center}
\vskip 0.37pc
\caption{The Pretzel knots $P(3,3,-3)$ and $P(9,3,-3)$.
\label{fig.color2}}
\end{figure}

\section{Quadratic forms from mapping classes}
\label{dsa2}
The purpose of this section is to define quadratic forms, which are isometric over $\Z$, from mapping classes in $\mathcal{M}_{g,1}$. 

We begin by reviewing the Dehn-Nielsen theorem. Let $F$ be $ \pi_1(\Sigma_{g,1})$. Choose a generating set $\{ x_1, \dots, x_{2g} \}$ of $F$, where $x_i$ is represented by the curve as illustrated in Figure \ref{fig.color33}. Take
\begin{equation}\label{m334} \zeta := [x_1, x_2] [ x_3, x_4] \cdots [x_{2g-1}, x_{2g}] \in \pi_1(\Sigma_{g,1}),
\end{equation}
which is represented by the boundary curve shown in Figure \ref{fig.color33}. Since $ f \in \mathcal{M}_{g,1} $ gives an automorphism on $ \pi_1(\Sigma_{g,1} )$, we have a group homomorphism $ \mathcal{M}_{g,1} \ra \mathrm{Aut}(\pi_1(\Sigma_{g,1} ))$. The Dehn-Nielsen-Baer theorem (see, e.g., \cite[Theorem 8.8]{FM}) claims that this is injective, and the image is $ \{ \ \tau \in \mathrm{Aut}(F ) \ | \ \tau(\zeta)=\zeta \}. $

\begin{figure}[htpb]
\begin{center}
\begin{picture}(100,50)
\put(72,3){\pc{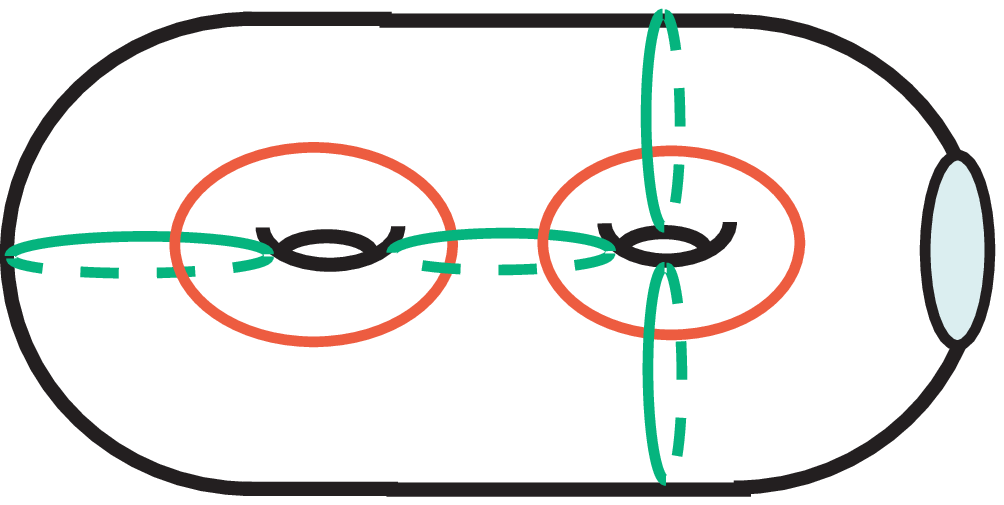}{0.3813}}
\put(-172,23){\pc{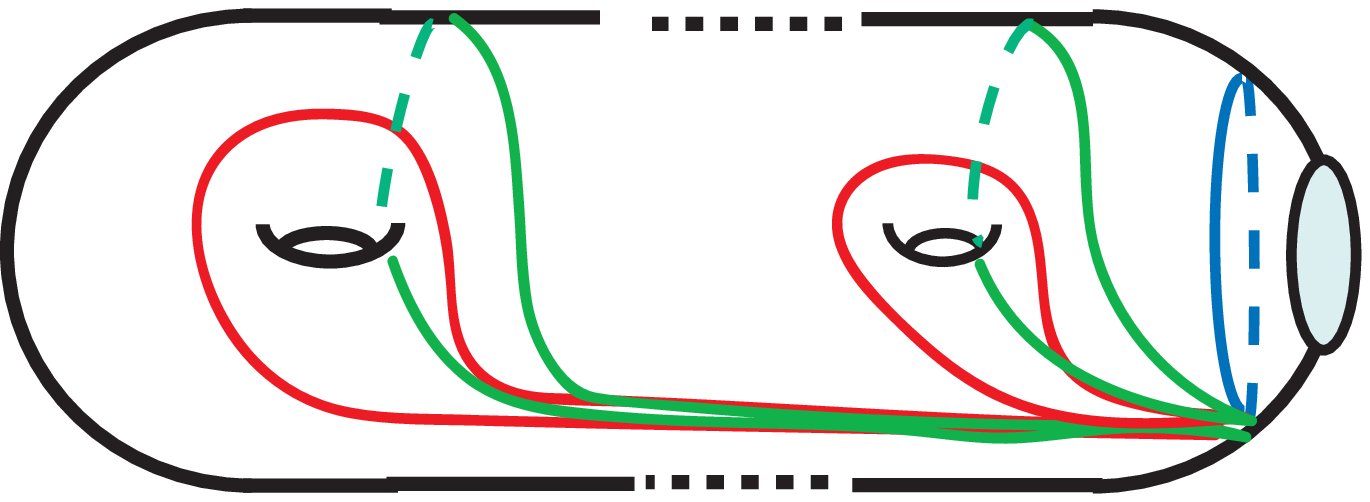}{0.41530174}}

\put(-153,29){\Large $x_1 $}
\put(-102,32){\Large $\cdots $}
\put(-122,52){\Large $x_2 $}
\put(-92,52){\Large $\cdots $}
\put(-46,51){\Large $x_{2g}$}
\put(-82,32){\Large $x_{2g-1} $}
\put(-16,42){\Large $\zeta$}
\put(84,9){\large $\alpha_1$}
\put(106,34){\large $\alpha_2$}
\put(126,11){\large $\alpha_3$}
\put(166,21){\large $\alpha_4$}
\put(153,36){\large $\alpha_5$}
\put(153,1){\large $\alpha_6$}

\end{picture}
\end{center}
\vskip 0.37pc
\caption{ The curves $x_i$ on the surface $\Sigma_{g,1}$, and curves $\alpha_i$ on the surface $\Sigma_{2,1}$,
\label{fig.color33}}
\end{figure}

Next, we will examine the mapping-torus $T_{f}$ with respect to $f \in\mathcal{M}_{g,1} .$ By a van Kampen argument, we notice the presentation,
\begin{equation}\label{m3as2} \pi_1(T_{f}) \cong \langle x_1, \dots, x_{2g} , \gamma \ \ | \ \ [x_1, \gamma] f_*(x_1) x_1^{-1}, \dots,[x_{2g}, \gamma] f_*(x_{2g}) x_{2g}^{-1}
\ \rangle.
\end{equation}
Here, $\gamma$ represents a generator of $\pi_1(S^1). $ The boundary of $ T_{f}$ is a torus, and $\pi_1( \partial (T_{f} )) \cong \Z^2 $ is generated by $\gamma$ and $ \zeta $. Since the 3-manifold $T_f$ is Haken, its homeomorphism type is determined by the triple of $ ( \pi_1(T_{f}) ,\gamma, \zeta )$ by the Waldhausen theorem \cite{Wal}. Thus, it is sensible to consider the relation between $\pi_1(T_{f}) $ and $\zeta $ in a meta-nilpotent quotient sense as follows:

Let us establish some more terminology. Assume that $f \in \mathrm{Aut}(F)$ arises from the mapping class in $\mathcal{M}_{g,1}$ via the Dehn-Nielsen embedding $\mathcal{M}_{g,1} \hookrightarrow \mathrm{Aut}(F)$. Then, the homology $ H_1( \Sigma_{g,1};\Z) \cong F/F_2$ is made into a torsion $\Z[t^{\pm 1}]$-module. The order of $F/F_2$ is called {\it the Alexander polynomial of $f$} and will be denoted by $\Delta_f$. As is well known (see, e.g., \cite{Tur}), the leading coefficient of $ \Delta_{f}$ is $\pm 1$ and $\mathrm{deg}(\Delta_{f})=2g $.

Similar to \eqref{bbb5}, let us consider the $\Z[t^{\pm 1}]$-module defined by the formula
$$ \mathcal{H}_f := \{ \Z\textrm{-equivariant homomorphism } \phi: F \ra F/F_2 \}. $$
As will be seen later in Lemma \ref{lem9}, for any $ \phi \in \mathcal{H}_f$, we have a lift $\Phi : F \ra F/F_3$ of $\phi $ as a $\Z\textrm{-equivariant homomorphism}$. Notice $ \Phi (\zeta ) \in F_2/F_3$, since $ \zeta \in F_2 $ by \eqref{m334}. Moreover, since $f(\zeta ) =\zeta $ by the Dehn-Nielsen theorem, $ \Phi (\zeta ) $ lies in the center $Z(0,1)\cap F_2/F_3 $ (cf. Lemma \ref{lem24}). Next, let us prove the following proposition.
\begin{prop}\label{b340}
Take $f \in \mathcal{M}_{g,1}  $ as above.
The map $\mathcal{Q}_f : \mathcal{H}_f \ra Z(0,1)\cap F_2/F_3 $ which sends $\phi$ to $\Phi (\zeta ) $ is a quadratic form on $\mathcal{H}_f$, and is independent of the choice of $\Phi$. Moreover, $\mathcal{Q}_f $ is isometric, that is, for any $x \in \mathcal{H}_f $, $\mathcal{Q}_f (f_*(x))= \mathcal{Q}_f (x). $
\end{prop}
\begin{proof}
For $ (a, \alpha ), (b, \beta ) \in F/F_2 \times F_2/F_3 $ the commutator $ [(a, \alpha ), (b, \beta ) ]  \in F_2/F_3$ equals $ a \sqcap b-b \sqcap a $ by definition \eqref{hhh}. Thus, $[\Phi(x_{2i}), \Phi(x_{2i+1})]= \phi(x_{2i}) \sqcap \phi (x_{2i+1})- \phi(x_{2i+1}) \sqcap \phi(x_{2i})$ as a quadratic form. Thus, $\Phi(\zeta ) = \prod_{i=1}^g [\Phi(x_{2i}), \Phi(x_{2i+1})]$ is also a quadratic form, and does not depend on the choice of $\Phi$.

Since $\Phi$ is $\Z$-equivariant and $f(\zeta)=\zeta$, we have $ \mathcal{Q}_f(f (\phi ))= f( \Phi(\zeta ))= \Phi ( f(\zeta))=\Phi ( \zeta )= \mathcal{Q}_f (\phi ) $, which means the isometry, as desired.
\end{proof}
\begin{defn}\label{b3440}
For $f \in \mathcal{M}_{g,1}$, we denote the quadratic map $\mathcal{H}_f \ra Z(0,1)\cap F_2/F_3 $ by $\mathcal{Q}_f$.
\end{defn}
\begin{rem}\label{b140}
Let $ p : \mathcal{M}_{g,1} \ra Sp(2g;\Z) $ be the canonical surjection, and let $\mathcal{T}_{g, 1}$ be the Torelli group, that is, $\Ker(p)$. If $f$ lies in $\mathcal{T}_{g, 1}$, the space $\mathcal{H}_f$ is trivial; thus, $\mathcal{Q}_f$ is trivial. Furthermore, we should notice that, if $f, f' \in \mathcal{M}_{g,1}$ are conjugate, $ \mathcal{Q}_f$ and $\mathcal{Q}_{f'} $ are equivalent by definition. However, there are $f_1, f_2 \in \mathcal{M}_{g,1}$ such that the classes $p(f_1), p(f_2) \in Sp(2g;\Z) $ are conjugate and $\mathcal{Q}_{f_1} $ and $\mathcal{Q}_{f_2}$ are not equivalent as follows: 
\end{rem}
\begin{exa}\label{b143}
For a simple closed curve $ \alpha_i \subset \Sigma_{2,1}$ in Figure \ref{fig.color33}, let $\tau(\alpha_i)$ be the positive Dehn twist along $\alpha_i $ in $\mathcal{M}_{2,1}$. Consider the mapping classes,
$$f_1 =\tau(\alpha_2) \tau(\alpha_3) \tau(\alpha_4)^{-1} \tau(\alpha_5) ^{-1}\tau(\alpha_1), \ \ \ \ f_2 =\tau(\alpha_2) \tau(\alpha_3) \tau(\alpha_3)\tau(\alpha_3)\tau(\alpha_4)^{-1} \tau(\alpha_5) ^{-1}\tau(\alpha_1) .$$
We can easily verify that the images $ p(f_1)$ and $p(f_2) $ in the symplectic group are conjugate. Furthermore, according to the knot info \cite{CL}, the mapping-tori $ T_{f_1}$ and $T_{f_2}$ are homeomorphic to the knot complements of $ 8_{20}$ and $12n_{582}$, respectively. Therefore, we will compute $\mathcal{Q}_{f_i} $ from the knot diagrams in a similar way to Examples \ref{ex5} and \ref{ex6}.

Let us compute the $\mathcal{Q}_{f_i} $. As in \cite{CL}, $F=H_1(\Sigma_{2,1};\Z)$ is a $\Z[t^{\pm 1}]$-module of the form $ \Z[t^{\pm 1}]/( (1-t+t^2)^2) $. Then, by Corollary \ref{gggg23}, the center $ Z(F/F_3 \rtimes \Z)$ is of rank 2 and generated by $C_1,C_2$. Let $ P_i : Z(F/F_3 \rtimes \Z)\ra \langle C_i \rangle $ the projection. With the help of a computer, the composite $ (P_1 +P_2) \circ \mathcal{Q}_{f_1}$ is shown to be the map,
$$(x +t y +t^2 z +t^3 w) \longmapsto w^2-2xw +x^2-wy+xy+y^2+wz-xz+yz+z^2, $$
and $ (P_1 +P_2)\circ \mathcal{Q}_{f_2}$ is equal to $3 (P_1 +P_2) \circ \mathcal{Q}_{f_1} $. 
However, 3 has no square-root in $ \Z[t^{\pm 1}]/( (1-t+t^2)^2) $. In summary, $\mathcal{Q}_{f_1} $ and $\mathcal{Q}_{f_2}$ turn out to be not equivalent.
\end{exa}

\section{Proofs of the statements in \S \ref{S54} }
\label{s23}
Here, we give the proofs of Theorem \ref{gggg2} and Theorem \ref{gggg4}.

We will suppose an automorphism $ \tau : F/F_3 \otimes \Q \stackrel{\sim }{\lra } F/F_3 \otimes \Q $ and will concretely express the action of $ \tau. $
For any $ X= (\sum_{i=1}^{m} {b_i} e_i , \ \sum_{i < j} c_{ij } e_{i j} ) \in F/F_3 \otimes \Q $, the target $\tau (X)$ can be described as
\begin{equation}\label{gggg256} \tau ((\sum_{i=1}^{m} {b_i} e_i), \ (\sum_{i < j} c_{ij } e_{i j} ) ) = ((\sum_{i=1}^{m} b_i' e_i), \ (\sum_{i< j} c_{ij }' e_{i j} ) ) \in F/ F_2 \otimes \Q \times F_2/F_3 \otimes \Q. \end{equation}
for some rational numbers $ b_i' $ and $ c_{ij }' $.
\begin{prop}\label{gggg}
Using the notation $b_i' $ and $ c_{ij }' $ in \eqref{gggg256}, let us regard $b_i' $ and $ c_{ij }' $ as functions with respect to ${b_i} , c_{ij } .$

Then, the $ b_i' $ is a linear function of the $b_i$'s, and the $ c_{ij }$'s are sums of a linear function of the $c_i $'s and a quadratic function of the $ b_i $'s.

Moreover, the restriction on $F_2/F_3 \otimes \Q$ is determined by the formula $\tau (0,x \sqcap y -y\sqcap x) = (0, tx \sqcap ty - ty \sqcap tx)$ for any $x, y \in F/F_2$.
\end{prop}

Before going the proof, let us show Theorem \ref{gggg2}. For this, we briefly review the Jordan decomposition over $\C$. For $\lambda \in \C \setminus \{0\}$ and $ \ell \in \N$, let $J_{\lambda}(\ell)$ be the Jordan block of size $\ell$ and eigenvalue $\lambda $. In other words, $J_{\lambda}(\ell) $ is $ \C[t]/(t-\lambda)^{\ell} $ as a $\C[t]$-module. According to \cite[Theorem 2]{MV},
the tensor product of $ J_{\lambda}(\ell)$ and $J_{\mu}(n)$ has the following decomposition:
\begin{equation}\label{pp9} J_{\lambda}(\ell) \otimes_{\C} J_{\mu}(n) \cong \bigoplus_{w=1}^{\mathrm{min}(\ell,n)} J_{ \ell \mu}(\ell + n- 2w-1) . \end{equation}
In addition, the exterior square $ J_{\lambda}(\ell) \bigwedge_{\C} J_{\lambda}(\ell)$ is a direct sum of $ J_{\lambda^2}(\dagger )$'s.
\begin{proof}[Proof of Theorem \ref{gggg2}.]
By Lemma \ref{lem24}, the center is isomorphic to the kernel of the linear map $ \mathrm{id}_{ F_2/F_3}- \tau|_{F_2/F_3}: F_2/F_3 \ra F_2/F_3$. In particular, the center is free.

Next, we will determine the rank of the kernel, which equals the rank of the center. For this, we denote by $\tau \otimes \C $ the extension of the map $\tau|_{F_2/F_3}$ to $F_2 /F_3 \otimes_{\Z}\C $. By Jordan decomposition of $ F_2 /F_3 \otimes_{\Z}\C$,
the desired rank of the center is equal to the number (i.e., multiplicity) of the Jordan blocks of eigenvalue $1 $. As a $ \C[t^{\pm 1}]$-submodule of $F_2 /F_3 \otimes_{\Z}\C $, let  $W_{\tau}$ be the direct sum of such Jordan blocks of eigenvalue $1$.
It is enough for the proof to determine the multiplicity of the direct sum.

We will compute the multiplicity.
Consider the subspace, $V$, of $F_1/F_2\otimes_{\Z} F_1 /F_2 \otimes_{\Z} \C $ generated by $(x\otimes y -y\otimes x) \otimes 1 $, and the $\tau$-equivariant linear map $\psi : V \ra F_2/F_3 \otimes \C $ which takes $(x\otimes y -y\otimes x) \otimes 1 $ to $ x\sqcap y -y\sqcap x$. By the definition of $\sqcap $ and Proposition \ref{gggg}, considering $(x,y)=(e_i,e_j)$ with $i<j $ can easily verify that $\psi $ is an isomorphism.
In summary, $F_2 /F_3 \otimes_{\Z} \C $ is linear isomorphic to the exterior square of $ F_1 /F_2 \otimes_{\Z} \C$.
As in \eqref{a5}, recall the Jordan decomposition of $ F_1 /F_2 \otimes_{\Z} \C$ as
$$ F/F_2 \otimes \C \cong \bigoplus_{i=1}^k \Bigl( J_{ \alpha_i}( n_{1}^{(i)} ) \oplus J_{ \alpha_i}( n_{2}^{(i)} ) \oplus \cdots \oplus J_{ \alpha_i}( n_{\ell_i}^{(i)} ) \Bigr). $$
Since the square $\alpha_i^2$ is not 1 by the assumption $\Delta_{\tau} (\pm 1) \neq 0$,
the submodule $W_{\tau}$ does not contain any direct summand of the form $ J_{\lambda}(\ell) \bigwedge_{\C} J_{\lambda}(\ell)$. Hence, this $W_{\tau}$ is a direct sum of $ J_{ \alpha_i}( n_{u}^{(i)} ) \otimes J_{ \alpha_i}( n_{v}^{(j)} ) $ such that $\alpha_i \alpha_j=1$ with $i<j $.
More precisely, it follows from \eqref{pp9} that the subspace $W_{\tau}$ is isomorphic to
$$ \bigoplus_{(i,j) \in \{ (i,j ) \in \N^2 | \ \alpha_i \alpha_j =1 , \ i< j \} } \bigoplus_{u=1} ^{\ell_i } \bigoplus_{v=1} ^{\ell_j } \bigoplus_{w=1} ^{ \mathrm{min} ( n_u^{ (i)} , n_v^{ (j)}) } J_{1} ( n_u^{ (i)}+ n_v^{ (j)}- 2w-1).$$
Notice that the multiplicity of $W_{\tau}$ is equal to \eqref{a77} exactly.
Since the multiplicity is equal to the rank of the center as mentioned above, we complete the proof.
\end{proof}

\begin{proof}[Proof of Proposition \ref{gggg}]
For $x \in F/F_2$ and $ \alpha \in F_2/F_3$, we can define $\kappa_{x}(\alpha) \in F_2/F_3$ such that $ \tau (x,\alpha)= (tx, \kappa_{x}(\alpha) ) $. Since $\tau$ is a group homomorphism, we have
\begin{equation}\label{pp}(tx+ ty, tx \sqcap ty + \kappa_{x}(\alpha) + \kappa_{y}(\beta) ) = (t x+ty, \kappa_{x +y }( x \sqcap y+ \alpha + \beta)) \end{equation}
for any $x ,y\in F/F_2, \alpha,\beta \in F_2/F_3$. When $x=0, \beta =0$, we have $ \kappa_{y}(\alpha ) = \kappa_y( 0) + \kappa_0( \alpha) $. Therefore, denoting $\kappa_y( 0) $ by $ \lambda ( y )$ and $\kappa_0( \alpha ) $ by $\eta(\alpha) $, it is enough for the proof to show that $\lambda$ is a quadratic form with respect to $y$ and $ \eta$ is a linear map with respect to $\alpha $.

We will determine $\eta $ first. Applying $y=-x$ and $\beta= - x \sqcap x$ to \eqref{pp}, we have
$$ 0= -t x \sqcap tx + \lambda ( -x) +\lambda (x) + \eta ( -x \sqcap x) \in F_2/F_3. $$
Then, by computing the commutator $\tau ( (x, 0)(y,0) (-x, x \sqcap x)(-y, y \sqcap y)) $, we obtain
\begin{equation}\label{pp2} \eta (x \sqcap y - y\sqcap x) = tx \sqcap ty -ty \sqcap tx. \end{equation}
Applying $x=y=0$ to \eqref{pp}, we should notice the linearity $ \eta( \alpha + \beta )= \eta ( \alpha) + \eta (\beta ) $. Since any $\alpha $ is a linear sum of $x \sqcap y - y\sqcap x$'s, $\eta$ is determined by \eqref{pp2} as desired.

Furthermore, from the linearity of $\eta$, \eqref{pp} is reduced to
\begin{equation}\label{pp4} \lambda(x+y) = \lambda (x) + \lambda(y) + tx \sqcap ty - \eta(x \sqcap y) \ \ \ \ \ \ \mathrm{for \ any \ } x, y \in F/F_2 \otimes \Q .\end{equation}
Notice that the map $\mathcal{B}: (F/F_2 \otimes \Q)^2 \ra \Q $ which sends $ (x,y) \mapsto tx \sqcap ty - \eta(x \sqcap y) $ is a symmetric bilinear map over $\Q$.
Let $B$ be an $(m\times m)$-matrix presentation of $\mathcal{B}$. Then, as is a common discussion on equivalence of symmetric bilinear maps and quadratic forms, $\lambda (x)$ must be $ x^T B x + A x $ for some $(m\times m)$-matrix $A$.
Namely, $\lambda $ is a quadratic function as required. 
\end{proof}
\noindent
As a corollary, in the situation of Theorem \ref{gggg4}, we now explicitly describe the target $\tau (X)$:
\begin{prop}\label{gggg3}
Suppose $n=1$, i.e., $ F/F_2 \otimes \Q \cong \Q[t ]/{f_1}(t)$, and that, if we expand $f_1$ as $ \sum_{i=0}^{2g} a_it^{i}$ with $a_0=a_{2g}=1$, they satisfy $ a_j=a_{2g- j}$ for any $0< j\leq g$. Let $b_{i}, c_{ij}'$ be the functions defined in \eqref{gggg256}.

Then, there are uniquely $ q_{i,j,k} \in \Q $ with $1 \leq j,k \leq 2g$ and $2 \leq i $ such that
\[ b_{1}'= - b_{2g} \textrm{\ \ \ \ and \ \ } b_{i}' = b_{i-1}-a_{i-1}b_{2g} \]
\[ c_{1j}' = ( b_{2g}^2 -b_{2g}) a_j /2 +c_{j,2g} - b_{2g} b_j + (q_{1,j,1} b_1 + \cdots + q_{1, j,2g} b_{2g}) , \]
\begin{equation}\label{pp7} c_{ij }' = (b_{2g}^2-b_{2g}) a_{i-1}a_{j-1}/2 +c_{i-1, j-1} - c_{i-1,2g} a_{j-1} +c_{j-1,2g} a_{i-1} - a_{j-1} b_{j-1} b_{2g}+ (q_{i,j,1} b_1 + \cdots + q_{i, j,2g} b_{2g}) . \end{equation}
\end{prop}
\begin{proof} 
Together with the above proof, it is enough to show that the correspondence \eqref{pp7} satisfies \eqref{pp}. Indeed, it is not so hard to check that $\tau (\sum_{i<j} c_{ij}e_i \sqcap e_j ) $ is equal to
$$ \sum_{i<j} c_{ij}(T e_i \sqcap Te_j - T e_j \sqcap T e_i)=
\sum_{j =2}^m c_{j, 2g}e_{1j} +\sum_{i=2}^m \sum_{j >i} (c_{i-1, j-1} - c_{i-1,2g} a_{j-1} +c_{j-1,2g} a_{i-1})e_{ij}.
$$
Therefore, we can show that the following is a solution of the equation \eqref{pp4}:
$$ \eta( \sum_{i=1}^{2g} b_{i}e_i )=\sum_{j =2}^{2g} \bigl(\frac{ (b_{2g}^2 -b_{2g}) a_j}{2} - b_{2g} b_j \bigr) e_{1j} +\sum_{i=2}^{2g} \sum_{j >i}\bigl( \frac{(b_{2g}^2-b_{2g}) a_{i-1}a_{j-1}}{2} - a_{j-1} b_{j-1} b_{2g} \bigr)e_{ij} . $$
Then, if we appropriately correct the first degree term of $\eta$, we can uniquely detect $\eta$
such that $\tau(x,\alpha)=(tx , \lambda (x)+ \eta (\alpha))$.
Hence, 
the proof is completed.
\end{proof}
We will give some examples:
\begin{exa}\label{gggg3rei}
If $g=1$, then
$$ \tau ( b_1e_1 +b_2 e_2 , c_1 e_{12}) = \bigl(-b_2 e_1 +(b_1 - b_2 a_1 )e_2, \ (a_1(b_2^2-b_2 )/2 + c_1 -b_1 b_2+q_{1,2,1} b_1 +q_{1,2,2} b_2 )e_{12} \bigr). $$
If $g=2$ and $f_1 =1+a_1t+ a_2 t^2 +a_1 t^3 +1$, the list of $c_{ij}'$ is as follows:
\normalsize
\[ \tau ( b_1 e_1 +b_2 e_2 +b_3 e_3 +b_4 e_4 , \ \ c_{12} e_{12}+ c_{13} e_{13}+ c_{14} e_{14}+c_{23} e_{23}+ c_{24} e_{24}+ c_{34} e_{34}) = \]
\[ \bigl( -b_4 e_1 + (b_1 -b_4 a_1)e_2 +( b_2 - b_4 a_2) e_3 +( b_3 - b_4 a_1) e_4 , \]
\[ (a_1 (b_4^2-b_4 )/2 +c_{14} - b_1 b_4 + b_1 q_{1,2,1} + b_2q_{1,2,2} + b_3q_{1,2,3} + b_4 q_{1,2,4 })e_{12 }\]
\[+ (a_2(b_4^2-b_4 )/2 +c_{24} - b_2 b_4 + b_1 q_{1,3,1} + b_2 q_{1,3,2} + b_3q_{1,3,3} +
b_4 q_{1,3,4} )e_{13} \]
\[+ (a_1(b_4^2-b_4 )/2 + c_{34} - b_3 b_4 + b_1 q_{1,4,1} + b_2 q_{1,4,2} + b_3q_{1,4,3} + b_4q_{1,4,4})e_{14}\]
\[ + (a_1a_2 (b_4^2-b_4 )/2 +c_{12} + c_{24} a_1 - b_2 b_4 a_2 - c_{14 }a_3 + b_1 q_{2,3,1} + b_2 q_{2,3,2} + b_3q_{2,3,3} + b_4 q_{2,3,4} )e_{23}\]
\[+(a_1^2(b_4^2-b_4 )/2 +c_{13 } + c_{34 }a_1 - b_3 b_4 a_1 - c_{14} a_1+ b_1 q_{2,4,1} + b_2 q_{2,4,2} + b_3q_{2,4,3} + b_4q_{2,4,4}) e_{24} \]
\[ +(a_1 a_2 (b_4^2-b_4 )/2+c_{23} + c_{34 }a_2 - b_3 b_4 a_2 - c_{24} a_1+ b_1 q_{3,4,1} + b_2 q_{3,4,2} + b_3q_{3,4,3} + b_4q_{3,4,4}) e_{34} \bigr)\]
\large
\end{exa}
\begin{proof}[Proof of Theorem \ref{gggg4}]
First, we claim $ d_{ i, 2g}^{(\ell)}= d_{1,i }^{(\ell)}$, that is, $\delta_{|g-i|,g-\ell} $. One observes that, if $ i \leq g$, the reciprocity $a_{2g-\ell+i}= a_{i-\ell }$ and \eqref{ooo} imply
$$d_{i, 2g}^{(\ell)}=d_{1,i }^{(\ell)}+ \sum_{j=1}^{i-1} (-a_{2g-j } \delta_{i-j,\ell }+ a_{i-j} \delta_{2g-j ,2g-\ell}) =d_{1,i }^{(\ell)}-a_{2g-\ell+i}+ a_{i-\ell }
= d_{1,i }. $$
On the other hand, if $ i>g $, we compute $d_{ i, 2g}^{(\ell)}- d_{1,i }^{(\ell)}$ as
\small
\[ \sum_{j=1}^{i-g} (a_{i-j} \delta_{2g-j,2g-\ell}-a_{2g-j} \delta_{i-j, 2g-\ell} )
+ \sum_{j=i-g+1}^{g} (a_{i-j} \delta_{2g-j,2g-\ell} -a_{2g-j} \delta_{i-j,\ell} )
+ \sum_{j=g+1}^{i-1} (a_{i-j} \delta_{2g-j,\ell} -a_{2g-j} \delta_{i-j,\ell} )
\]
\large
\[= \sum_{j=1}^{g} a_{i-j} \delta_{2g-j,2g-\ell} - \sum_{j= 1}^{i-g} a_{2g-j} \delta_{i-j, 2g-\ell}
- \sum_{j=i-g+1}^{i-1} a_{2g-j} \delta_{i-j,\ell}
+ \sum_{j=g+1}^{i-1} a_{i-j} \delta_{2g-j,\ell}
\]
\[= a_{i-\ell} - 0- a_{2g-\ell+i} + 0 =0,\]
which proves the claim. Next, let $d_{i,j}''$ be the $e_{i,j}$-th component of $ \tau(C_\ell)$. Then, Proposition \ref{gggg3} immediately leads to
\[ d_{1,j }''= d_{j, 2g, }^{(\ell)}, \mathrm{ \ \ \ and \ \ \ }d_{i,j}''= d_{i-1,j-1 }^{(\ell)}-a_{j-1}d_{i-1, 2g}^{(\ell)} + a_{i-1} d_{j-1, 2g}^{(\ell)} \ \ \ \mathrm{if} \ \ i >1. \]
An inductive discussion easily leads to $d_{i,j}''=d_{i,j}^{(\ell)}$ by the definition of $c_{i,j}$ and the above claim. Hence, we have $\tau(C_\ell)=C_{\ell}$, as desired. Since $C_\ell \in F_2/F_3$, $C_k$ lies in the center of $F/F_3 \rtimes \Z $.

Finally, we will show the latter part.
By Lemma \ref{lem24}, the center is isomorphic to the kernel of the linear map $\tau|_{F_2/F_3 \otimes \Q}: F_2/F_3 \otimes \Q \ra F_2/F_3 \otimes \Q$. In particular, the center is a $\Q$-vector space. By the proof of Theorem \ref{gggg2}, the dimension is $g$.
Since $C_1, \dots, C_g$ are linearly independent as in the proof of Corollary \ref{gggg45533}, the vector space is spanned by them, as desired.
\end{proof}

\section{Proofs of Propositions in \S \ref{dsa}}
\label{k2k2}
To begin with, we need the following lemma in order to prove Proposition \ref{lem6}:
\begin{lem}\label{lem9}
Let $G$ and $K$ be groups acted on by $\Z$. Suppose that the group homology $H_2^{\rm gr}(G;\Z) $ and $H_3^{\rm gr}(G \rtimes \Z ;\Z) $ are zero. Let $p: \overline{K} \ra K$ be a  central extension of $ K$ whose kernel $ \Ker (p)$ is isomorphic to $ H_2(K;\Q)$. 

Then, any $\Z$-equivariant homomorphism $f: G \ra K $ admits a $\Z$-equivariant homomorphism $\tilde{f} : G \ra \overline{K}$ as a lift of $f$.
\end{lem}
\begin{proof}
Consider the pullback, $ \overline{G}$, of $f$ and $p$, and take $ \bar{f}: \overline{G} \ra \overline{K} $. Since $H_2^{\rm gr}(G;\Z) =0$ and $\Ker (p )$ is uniquely divisible, we have $ H^2_{\rm gr}(G; \Ker (p )) \cong 0$ by the universal coefficient theorem. Thus, there is a group isomorphism $\overline{G} \cong G \times \Ker(p) $.

Here, we claim that this isomorphism is $\Z$-equivariant. Considering the Lyndon-Hochschild spectral sequence from $ G \ra G \rtimes \Z \ra \Z $, we have $ H^2_{\rm gr}( G ; \Ker (p))^{\Z} \cong H^3_{\rm gr}( G ; \Ker (p))$, which is zero by assumption. Thus, since we can choose a $\Z$-equivariant section $G \ra G \times \Ker(p) $, 
the action of $\Z$ on $G \times \Ker(p) $ diagonally. 

Hence, the composite of the inclusion $ G \hookrightarrow \overline{G}$ and $\bar{f} : \overline{G}\ra \overline{K} $ is a desired lift.
\end{proof}
\begin{proof}[Proof of Proposition \ref{lem6}]
We will give an inductive proof on $k$. Since the case with $k=2$ is trivial, we may assume $k>2$. Let $\overline{K}$ be $F/F_{k+1} \otimes \Q $ and $ K$ be $F/F_{k} \otimes \Q $.

Fix an isomorphism $ \pi_K \cong [\pi_K , \pi_K ] \rtimes \Z .$ Since $S^3 \setminus K $ is an Eilenberg-MacLane space, so is the covering space $\widetilde{E}_K$. Thus, the group homology of $ \pi_1( \widetilde{E}_K) =[\pi_K, \pi_K]$ is the homology of $\widetilde{E}_K $; we notice $H_2^{\rm gr}( \pi_1( \widetilde{E}_K) ;\Z) \cong H_2( \widetilde{E}_K;\Z) =0$ and $H_3( \pi_K ;\Z) \cong H_3( S^3 \setminus K ;\Z) \cong 0$.

Thanks to Lemma \ref{lem9}, from any $\Z$-equivariant map $[\pi_K, \pi_K] \ra F/F_k \otimes \Q $ as the restriction of $f : \pi_K \ra (F/F_k \otimes \Q ) \rtimes \Z $, we have a lift $[\pi_K, \pi_K] \ra F/F_{k+1} \otimes \Q $. Since this lift is $\Z$-equivariant by Lemma \ref{lem9}, we have a homomorphism $ \widetilde{f} : \pi_K\ra (F/F_3 \otimes \Q ) \rtimes \Z$, as required.

Finally, we show the uniqueness of $\widetilde{f}$. From the proof of Lemma \ref{lem9}, another lift $\widetilde{f}'$ bijectively corresponds to the choice of the inclusion $ G \hookrightarrow \overline{G}$. Thus, we can uniquely find $ \widetilde{f}$ with $ \widetilde{f}(\mathfrak{m})=(0,1) \in (F/F_k \otimes \Q) \rtimes \Z $, as required.
\end{proof}

\begin{proof}[Proof of Proposition \ref{lem7} (I)]
Notice that $\mathfrak{l} \in H_1 ( [\pi_K, \pi_K ];\Z ) $ is zero, because $\mathfrak{l} $ is bounded by a Seifert surface of $\widetilde{E}_K$; thus, $ \widetilde{f} (\mathfrak{l} ) $ lies in $F_2/F_3 \otimes \Q $. Since $ \mathfrak{l} $ commutes with the meridian, $\widetilde{f} (\mathfrak{l} ) \in Z(e ,1)$. By Lemma \ref{lem24}, $\widetilde{f} (\mathfrak{l} ) $ lies in the center, as desired. We show (II) later.
\end{proof}

\begin{proof}[Proof of Proposition \ref{lem8}]
Since the meridian-longitude pair $(\mathfrak{m}, \mathfrak{l})$ commutes, $\mathcal{Q}_K ( \tau_* f) = \widetilde{f}(\mathfrak{m}^{-1} \mathfrak{l}\mathfrak{m} ) = \widetilde{f}( \mathfrak{l} ) = \mathcal{Q}_K ( f) $ as required.

Next, we will show that $\mathcal{Q}_K$ is a quadratic form. Choose a knot diagram $D$. Let $\gamma_0, \gamma_1, \dots, \gamma_m$ be the arcs of $D.$ Here we may assume that the meridian $\mathfrak{m}$ is represented by a loop circulating around $\gamma_0$. Then, given $f \in \mathcal{H}_K$, the Wirtinger presentation implies the correspondence $\gamma_k \mapsto ( a_k ,1) \in F/F_2 \rtimes \Z $ for some $a_k \in F/F_2 $, where we should notice $a_0=0.$
Then, from the Wirtinger presentation, the crossing as in Figure \ref{fig.color} requires the relation $ f( \gamma_k ) = f( \gamma_j )^{-1}f( \gamma_i )f ( \gamma_j) $, which means linear equations with respect to $a_\ell$'s and that $ \mathcal{H}_K$ can be regarded as a space of the linear equations.

In addition, consider the lift $\widetilde{f}$ of $f$. There are $b_k \in F_2/F_3 \otimes \Q$ such that
$$\widetilde{f}( \gamma_k ) =( b_k, a_k ,1 ) \in(F_2/F_3 \otimes \Q) \times (F/F_2 \otimes \Q) \times \Z . $$
Similarly, for the crossing as in Figure \ref{fig.color}, we have the relation $\widetilde{f}( \gamma_k ) = \widetilde{f}( \gamma_j )^{-1}\widetilde{f}( \gamma_i )\widetilde{f}( \gamma_j) $, which equivalently means,
\begin{equation}\label{pp0} ( b_k, a_k ,1 ) = ( \eta ( b_i - b_j ) + b_j + \eta^2 ( a_j \sqcap a_j ) + \eta ( \lambda(-a_j))+ \lambda(a_i - a_j) , \ t( a_i - a_j) +a_j , \ 1). \end{equation}
Here, $\lambda $ and $\eta $ are the functions used in the proof of Proposition \ref{gggg}. All such equations give simultaneous linear equations with respect to $b_k$'s. Proposition \ref{lem6} ensures a unique equation on $b_k$'s in terms of $a_k$'s. Since $\lambda$ is a linear function of $b_k$'s, and $\eta$ is a quadratic function of $a_k$'s from the proof of Proposition \ref{gggg}, so are $ b_k$'s.

The preferred longitude $\mathfrak{l} \in \pi_1(S^3 \setminus K)$ is some product of $ \gamma_k$'s. Thus, the discussion in the above paragraph tells us that the target $\widetilde{f}(\mathfrak{l}) \in (F_2/F_3 \otimes \Q)\rtimes \Z$ can be written as a quadratic function of $f(\gamma_k)$'s. This means a quadratic function on $\mathcal{H}_K.$
\end{proof}

\begin{figure}[htpb]
\begin{center}
\begin{picture}(100,50)
\put(-102,52){\Large $\gamma_i $}
\put(-46,51){\Large $\gamma_j $}
\put(-46,18){\Large $\gamma_k $}
\put(-62,123){\pc{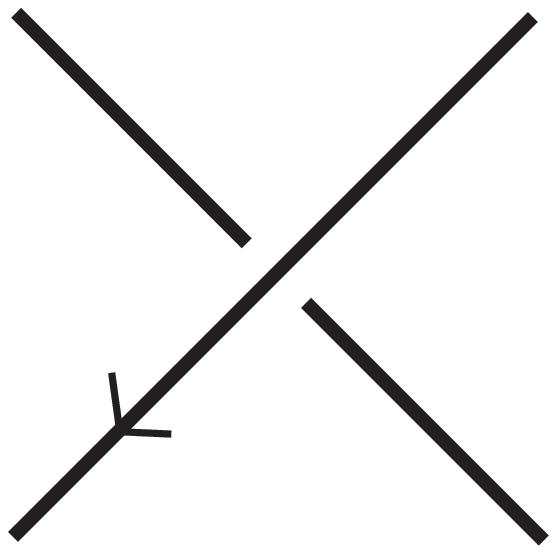}{0.2530174}}
\put(46,51){\large $\widetilde{f} (\gamma_i )= ( b_i,a_i,1), $}
\put(46,30){\large $\widetilde{f} (\gamma_j )= ( b_j,a_j,1) ,$}
\put(46,9){\large $\widetilde{f} (\gamma_k )= ( b_k,a_k, 1) \in (F/F_3 \otimes \Q)\rtimes \Z.$}

\end{picture}
\end{center}
\vskip -1.9937pc
\caption{The three arcs around a crossing.
\label{fig.color}}
\end{figure}

\begin{proof}[Proof of Proposition \ref{lem7} (II)]
For two lifts $\tau$ and $\tau'$ of $t$, we have the associated lifts $\widetilde{f}$ and $\widetilde{f}' $ of $f$, respectively; we will show $\widetilde{f} (\mathfrak{l}) = \widetilde{f} '(\mathfrak{l})$. Let us consider the map,
\[ \Upsilon: (F/F_3 \otimes \Q \times \Z)^2 \lra F/F_3 \otimes \Q \times \Z; \ \ \ \ (a,n,b,m) \longmapsto (b^{-1}a, \mathrm{max}(m,n))\]
and the correspondence 
$$g: \gamma_k \mapsto \Upsilon (\widetilde{f}(\gamma_k) , \widetilde{f} '(\gamma_k) ).$$ Then, from \eqref{pp0} we obtain $ g(\gamma_k)=g(\gamma_j)^{-1}g(\gamma_i)g(\gamma_j)$, or equivalently,
$$ (b_k -b_k',0,1)= (\eta (b_i -b_i' +b_j' -b_j)+ b_j- b_j' ,0,1) \in (F_2/F_3 \otimes \Q \times \{0 \})\rtimes \Z$$
Therefore, it can be easily verified that this $g $ gives rise to a homomorphism $ g: \pi_1(S^3 \setminus K) \ra (F_2/F_3 \otimes \Q)\rtimes \Z $. Considering the metabelian quotient of $g$, this $g$ factors through $H_1(\widetilde{E}_K;\Q )\rtimes \Z $. As mentioned above, the class of $ \mathfrak{l} $ in $H_1(\widetilde{E}_K;\Q )\rtimes \Z $ is zero; hence, $g (\mathfrak{l}) =0$. Since $g (\mathfrak{l})= \widetilde{f} (\mathfrak{l}) -\widetilde{f} '(\mathfrak{l}) $ by definition, we have $\widetilde{f} (\mathfrak{l}) = \widetilde{f} '(\mathfrak{l})$, as desired.
\end{proof}

\subsection*{Acknowledgments}
The author expresses his gratitude to Takahiro Kitayama for many helpful discussions.
He also thanks the referee for carefully reading this paper.

\vskip 1pc

\normalsize

DEPARTMENT OF MATHEMATICS TOKYO INSTITUTE OF TECHNOLOGY 2-12-1 OOKAYAMA, MEGURO-KU, TOKYO 152-8551 JAPAN

\end{document}